\documentclass[10pt]{article}
\usepackage{amssymb,amsmath,color}
\usepackage{amsthm}
\usepackage{mathrsfs}
\usepackage{fullpage}

\allowdisplaybreaks
\begin{document}

\def\pd#1#2{\frac{\partial#1}{\partial#2}}
\def\dfrac{\displaystyle\frac}
\let\oldsection\section
\renewcommand\section{\setcounter{equation}{0}\oldsection}
\renewcommand\thesection{\arabic{section}}
\renewcommand\theequation{\thesection.\arabic{equation}}

\def\Xint#1{\mathchoice
  {\XXint\displaystyle\textstyle{#1}}%
  {\XXint\textstyle\scriptstyle{#1}}%
  {\XXint\scriptstyle\scriptscriptstyle{#1}}%
  {\XXint\scriptscriptstyle\scriptscriptstyle{#1}}%
  \!\int}
\def\XXint#1#2#3{{\setbox0=\hbox{$#1{#2#3}{\int}$}
  \vcenter{\hbox{$#2#3$}}\kern-.5\wd0}}
\def\ddashint{\Xint=}
\def\dashint{\Xint-}

\newcommand{\cblue}{\color{blue}}
\newcommand{\cred}{\color{red}}
\newcommand{\be}{\begin{equation}\label}
\newcommand{\ee}{\end{equation}}
\newcommand{\beaa}{\begin{eqnarray}}
\newcommand{\bea}{\begin{eqnarray}\label}
\newcommand{\eea}{\end{eqnarray}}
\newcommand{\nn}{\nonumber}
\newcommand{\intO}{\int_\Omega}
\newcommand{\Om}{\Omega}
\newcommand{\cd}{\cdot}
\newcommand{\pa}{\partial}
\newcommand{\ep}{\varepsilon}
\newcommand{\uep}{u_{\eta}}
\newcommand{\vep}{v_{\eta}}
\newcommand{\Sep}{S_{\eta}}
\newcommand{\Lp}{L^p(\Om)}
\newcommand{\Lq}{L^q(\Om)}
\newcommand{\abs}{\\[2mm]}
\newcommand{\eps}{\varepsilon}
\newcommand{\norm}[2][ ]{\|#2\|_{#1}}
\newtheorem{thm}{Theorem}
\newtheorem{lem}{Lemma}[section]
\newtheorem{proposition}{Proposition}[section]
\newtheorem{dnt}{Definition}[section]
\newtheorem{remark}{Remark}[section]
\newtheorem{cor}{Corollary}[section]
\allowdisplaybreaks

\title{Boundedness in a three-dimensional chemotaxis-haptotaxis model}
\author{
Xinru Cao\footnote{caoxinru@gmail.com, supported by the Fundamental Research Funds for the Central Universities, and the Research Funds of Renmin University of China}\\
{\small Institute for Mathematical Sciences, Renmin University of China, 100872 Beijing, China}
}

\maketitle
\date{}
\begin{abstract}
This paper studies the chemotaxis-haptotaxis system
\beaa\nn
\left\{
\begin{array}{llc}
u_t=\Delta u-\chi\nabla\cdot(u\nabla v)-\xi\nabla\cdot(u\nabla w)+\mu u(1-u-w), &(x,t)\in \Omega\times (0,T),\\
v_t=\Delta v-v+u, &(x,t)\in\Omega\times (0,T),\\
w_t=-vw,&(x,t)\in \Omega\times (0,T)
\end{array}
\right.\quad\quad(\star)
\eea
under Neumann boundary conditions. Here $\Om\subset\mathbb{R}^3$ is a bounded domain with smooth boundary
and the parameters $\xi,\chi,\mu>0$. We prove that for nonnegative and suitably smooth initial data $(u_0,v_0,w_0)$, if $\chi/\mu$
is sufficiently small, ($\star$) possesses a global classical solution which is bounded in $\Om\times(0,\infty)$.
We underline that the result fully parallels the corresponding parabolic-elliptic-ODE system.
\end{abstract}

\section{Introduction}

The motion of cells moving towards the higher concentration of a chemical signal is called chemotaxis.
A classical mathematical model for this type of processes was initiated by Keller and Segel in the 1970's \cite{keller_segel},
and in a prototypical form this model writes
\beaa
\left\{
\begin{array}{llc}
u_t=\Delta u-\chi\nabla\cdot(u\nabla v), &(x,t)\in \Omega\times (0,T),\\
\displaystyle
v_t=\Delta v-v+u, &(x,t)\in\Omega\times (0,T).
\end{array}
\right.
\eea
Here $u$ is the density of cells, $v$ denotes the concentration of the chemical substance, and $\chi>0$
measures the sensitivity of chemotactic response. In recent 40 years,
a large quantity of literature has been devoted to study the global existence as well as singularity formation
in either finite or infinite time for this system \cite{Nagai-Senba-Yoshida,HW,win_JMPA}.
We also refer to \cite{H,HP,BBTW} for a broad overview.\\
Apart from this classical Keller-Segel system, a large number of variants has been proposed to describe taxis phenomena
in mathematical biology.
Among them, a model for tumor invasion mechanism was introduced by Chaplain and Lolas \cite{chan_lolas}.
In this model, tumor cells are assumed to produce a diffusive chemical substance,
the so-called matrix-degrading enzyme (MDE), which decays non-diffusive
static healthy tissue (ECM). It is observed that both the enzyme and the healthy tissue can attract the cancer cells
in the sense that the cancer cells bias their movement along the gradients of the concentrations of both ECM and MDE,
where the former of these processes, namely taxis toward a non-diffusible quantity, is usually referred as haptotaxis.\\
Additionally, the cancer cells compete for space with ECM,
and at the considered time scales moreover logistic-type cell kinetics need to be taken into account.
If futhermore the ability of ECM to spontaneously renew is included, the Chaplain-Lolas model becomes
\be{1.0}
\left\{
\begin{array}{llc}
u_t=\Delta u-\chi\nabla\cdot(u\nabla v)-\xi\nabla\cdot(u\nabla w)+\mu u(1-u-w),
&(x,t)\in \Omega\times (0,T),\\
\tau v_t=\Delta v-v+u, &(x,t)\in\Omega\times (0,T),\\
w_t=-vw+\eta w(1-u-w),&(x,t)\in \Omega\times (0,T),
\end{array}
\right.
\ee
where $u$, $v$ and $w$ denote the density of cells, the concentration of MDE and the density of ECM, respectively,
where the parameters $\xi$, $\chi$, $\mu$, $\eta$ are positive constants and $\tau \ge 0$,
and where $\Om\subset\mathbb{R}^N$, $N\ge 1$, denotes the physical domain under consideration.\\
Assuming $w\equiv 0$, (\ref{1.0}) is reduced to the classical Keller-Segel system with logistic source,
which has extensively been studied during the past 20 years. Compared with the pure chemotaxis system mentioned above,
one may expect the logistic source and, especially, death terms to enhance the possibility of bounded solutions.
In fact, Tello and Winkler \cite{tello_win_CPDE} proved that if $\tau=0$ and
\bea{mu}
\mu>\frac{(N-2)_+}{N}\chi,
\eea
then for any regular initial data, the system admits a unique global classical solution which is bounded.
In the case $\tau=1$, it is known that bounded solutions exist in lower dimensions ($N=1,2$) for any $\mu>0$
\cite{OTYM},
and that the same result holds for $\mu>\mu_0$ with some $\mu_0(\chi)>0$ in higher dimensions \cite{W1}.
More precisely, a careful inspection of the proofs therein shows that in fact
large values of the ratio $\frac{\mu}{\chi^2}$ are sufficient to exclude blow up in either finite time or infinite time.\\
Concerning (\ref{1.0}) with possibly nontrivial $w$,
the strong coupling between remodeling and chemotaxis substantially complicates the situation, and accordingly
the knowledge on this topic is quite incomplete so far.
To the best of our knowledge, global existence of weak solution is obtained in \cite{stinner_surulescu_win} for $N\le 3$, where (\ref{1.0}) is included as a subsystem. And global solvability of classical solutions in this full system is known only when $\tau=0$ and $N=2$
\cite{tao_win_jde}.
Disregarding the chemotaxis effect, the haptotaxis-only version with $\chi=0$, $\tau=1$ was studied in
\cite{tao_nonanal}.\\
In the real situations, the ECM degrades much faster than it renews, thus the remodeling effect can be neglected,
that is, we may assume $\eta=0$.
Under this hypothesis, the corresponding parabolic-elliptic simplification $\tau=0$
has been studied by Tao and Winkler in \cite{Tao_Win_non}, where it has been proved that solutions stay bounded under
the same condition as in the case $w\equiv 0$, that is, when (\ref{mu}) holds. This shows that in this situation
the haptotaxis term does not affect the boundedness of solution,
and that accordingly the chemotaxis process essentially dominates the whole system. A natural question is whether
a similar conclusion holds in the fully parabolic system obtained on letting $\tau=1$.
In \cite{Tao_2D}, Tao gives a partially positive answer in this direction by
proving that when $N=2$, solutions remain bounded for any $\mu>0$, which thus parallels known results both for $\tau=0$,
and also for $\tau=1$ when $w\equiv 0$.
As far as we can tell, however, despite a result on global existence
established in \cite{tao_wang_non}, the question of boundedness of solutions is completely open in higher dimensions.
It is the purpose of this work to furthermore establish a corresponding parallel result for the three-dimensional
parabolic-parabolic-ODE chemotaxis-haptotaxis model in this direction.\\
Accordingly, we deal with the system
\be{2.0}
\left\{
\begin{array}{llc}
u_t=\Delta u-\chi\nabla\cdot(u\nabla v)-\xi\nabla\cdot(u\nabla w)+\mu u(1-u-w),
&(x,t)\in \Omega\times (0,T),\\[6pt]
v_t=\Delta v-v+u, &(x,t)\in\Omega\times (0,T),\\[6pt]
w_t=-vw,&(x,t)\in \Omega\times (0,T),\\[6pt]
\frac{\partial u}{\partial\nu}-\chi u\frac{\partial v}{\partial\nu}-\xi u\frac{\partial w}{\partial\nu}
=\frac{\partial v}{\partial \nu}=0,&(x,t)\in \partial\Omega\times (0,T),\\[6pt]
u(x,0)=u_0(x),\;\;v(x,0)=v_0(x),\;\;w(x,0)=w_0(x), &x\in \Omega,
\end{array}
\right.
\ee
where $\Om\subset\mathbb{R}^3$ is bounded with smooth boundary and $\chi,\xi,\mu>0$.
We assume that initial data are regular enough and satisfy a standard compatibility condition in the sense that
\beaa
\left\{
\begin{array}{llc}\label{inidata}
&u_0\in C^0(\bar{\Om}),\;\;v_0\in W^{1,\infty}(\Om), \;\;w_0\in C^{2,\alpha}\;(\bar{\Om}) (\alpha\in(0,1)),\\[6pt]
&\frac{\partial w_0}{\partial\nu}=0.
\end{array}
\right.
\eea
Then our main result says the following.
\begin{thm}\label{th1}
  There exists $\theta_0>0$ such that whenever $\chi>0$, $\mu>0$ and $\xi>0$ are such that $\frac{\chi}{\mu}<\theta_0$,
  for any initial data $(u_0,v_0,w_0)$ fulfilling
  \eqref{inidata}, there exists a unique classical solution $(u,v,w)$,
  which is global in time and bounded in $\Om\times(0,\infty)$.
\end{thm}

\begin{remark}
We only carry out the proofs for three-dimensional case in this paper. Actually, we can start the iteration from the higher regularity of $v\in L^{\frac{2N}{N-2}}(\Om)$, the same result still holds for $N\le 15$.
\end{remark}

We see that although our hypothesis on the parameters is not as explicit as (\ref{mu}) obtained for the
parabolic-elliptic counterpart,
it still shows that again boundedness of solutions is enforced by a condition merely referring to the interplay between
chemotaxis and quadratic degradation in logistic source.\\
Apart from this, we find it worth mentioning that our approach even shows a new result for the pure fully parabolic
chemotaxis system with
logistic source in the sense that when $w\equiv 0$, $N\ge 3$, the system admits a classical bounded
solution if $\frac{\mu}{\chi}$ is sufficiently large.
Compared with a similar conclusion under the alternative assumption that $\frac{\mu}{\chi^2}$ be large \cite{W1},
our result seems more consistent
with (\ref{mu}) for the parabolic-elliptic system where the linear ratio $\frac{\mu}{\chi}$
is found to determine the boundedness of solution.
\section{Preliminaries}
We first state a result on local existence of classical solutions.
Without essential difficulties, the proof can be derived based on that in \cite[Lemma 2.1]{tao_win_pe}.
\begin{lem}\label{loc}
  Let $N=3$, $\chi>0$, $\xi>0$ and $\mu>0$. For $(u_0,v_0,w_0)$ satisfying (\ref{inidata}),
  there is $T_{\max}\in(0,\infty]$ such that (\ref{2.0}) admits a unique classical solution
  \begin{align}\nn
	&u\in C^0(\bar{\Om}\times[0,T_{\max}))\cap C^{2,1}(\bar{\Om}\times(0,T_{\max})),\\\nn
	&v\in C^0(\bar{\Om}\times[0,T_{\max}))\cap C^{2,1}(\bar{\Om}\times(0,T_{\max}))\cap
    L_{loc}^\infty([0,T_{\max});W^{1,q}(\Om))(q>3),\\\nn
	&w\in C^{2,1}(\bar{\Om}\times[0,T_{\max})),
  \end{align}
  such that
  \begin{align}
  u\ge 0,\quad v\ge 0 \quad \text{ and } \quad 0<w\le \|w_0\|_{L^\infty(\Om)} \quad\text{ for all }\quad t\in[0,T_{\max}).
  \end{align}
  Moreover, if $T_{\max}<\infty$, then
  \begin{align}
	\mathop{\limsup}\limits_{{t\nearrow T_{\rm max}}}
	(\|u(\cdot,t)\|_{L^{\infty}(\Omega)})=\infty.
  \end{align}
\end{lem}
According to the above existence theory, we know that if we fix any $s_0\in(0,T_{\max})$,
then there exists  $M>0$ such that
\begin{align}\label{s0}
	\|u(\cdot,s_0)\|_{L^{\infty}(\Om)}
	+\|v(\cdot,s_0)\|_{L^{\infty}(\Om)}
	+\|w(\cdot,s_0)\|_{W^{2,\infty}(\Om)}<M.
\end{align}
Observing that $w$ can be represented by $v$ and $w(x,s_0)$, we can compute $\Delta w$ in a convenient way.
Upon a slight adaptation of \cite[Lemma 2.2]{Tao_Win_non}, we can prove a one-sided pointwise estimate for $\Delta w$
as follows.
\begin{lem}\label{delta w}
  Let $(u_0,v_0,w_0)$ satisfy \eqref{inidata}, $(u,v,w)$ solve \eqref{2.0}. For all $s_0\ge0$, we have
  \beaa\nn
	\Delta w(x,t)&\ge&\Delta w(x,s_0)\cdot e^{-\int_{s_0}^tv(x,s)ds}
	-2e^{-\int_{s_0}^t v(x,s)ds}\nabla w(x,s_0)\cdot\int_{s_0}^t\nabla v(x,s)ds\\\label{2.2.0}
	&&-\frac{1}{e}w(x,s_0)-w(x,s_0)v(x,t)e^{-\int_{s_0}^t v(x,s)ds}
  \eea
for all $x\in\Om$ and all $t\in(s_0,T_{\max})$.
\end{lem}
\begin{proof}
  Representing $w(x,t)$ according to
  \bea{2.1}
	w=e^{-\int_{s_0}^t v(x,s)ds}w(x,s_0)
  \eea
  for all $x\in\Omega$ and $t\in(s_0,T_{\max})$,
  we directly compute
  \beaa\nn
	\Delta w(x,t)
	&=&
	\Delta w(x,s_0)e^{-\int_{s_0}^tv(x,s)ds}-2e^{-\int_{s_0}^t v(x,s)ds}\nabla w(x,s_0)\cdot\int_{s_0}^t\nabla v(x,s)ds\\\nn
	&&+w(x,s_0)e^{-\int_{s_0}^t v(x,s)ds}|\int_{s_0}^t\nabla v(x,s)ds|^2-w(x,s_0)e^{-\int_{s_0}^t v(x,s)ds}
	\int_{s_0}^t\Delta v(x,s)ds.
  \eea
  Since $ze^{-z}\le\frac{1}{e}$, by dropping some nonnegative terms, we obtain that
  \beaa\nn
	\Delta w(x,t)&\ge&
	\Delta w(x,s_0)e^{-\int_{s_0}^tv(x,s)ds}-2e^{-\int_{s_0}^t v(x,s)ds}\nabla w(x,s_0)
	\cdot\int_{s_0}^t\nabla v(x,s)ds\\\nn
	&&~~~~~~~~-w(x,s_0)e^{-\int_{s_0}^tv(x,s)ds}\int_{s_0}^t(v_s(x,s)+v(x,s)-u(x,s))\\\nn
	&\ge&
	\Delta w(x,s_0)e^{-\int_{s_0}^tv(x,s)ds}-2e^{-\int_{s_0}^t v(x,s)ds}\nabla w(x,s_0)
	\cdot\int_{s_0}^t\nabla v(x,s)ds\\\nn
	&&~~~~~~~~-w(x,s_0)e^{-\int_{s_0}^t v(x,s)ds}(v(x,t)-v(x,s_0))-w(x,s_0)
	e^{-\int_{s_0}^t v(x,s)ds}\int_0^tv(x,s)ds\\\nn
	&\ge&
	\Delta w(x,s_0)e^{-\int_{s_0}^tv(x,s)ds}-2e^{-\int_{s_0}^t v(x,s)ds}\nabla w(x,s_0)\cdot\int_0^t\nabla v(x,s)ds\\\nn
	&&~~~~~~~~-w(x,s_0)v(x,t)e^{-\int_{s_0}^t v(x,s)ds}-\frac{1}{e}w(x,s_0)
  \eea
 for all $t\in(s_0,T_{\max})$. Thus the proof is complete.
\end{proof}
With the aid of Lemma \ref{delta w}, we can furthermore prepare a preliminary estimate of an integral
related to the haptotactic interaction.
This estimate will be used in different ways later on.
\begin{lem}\label{lem2.3}
  Let $\chi>0$, $\xi>0$, and assume that \eqref{inidata} holds. Then for any $p>1$, $s_0\in(0,T_{\max})$, the solution
  of (\ref{2.0}) satisfies
  \bea{2.3}
	(p-1)\xi\intO u^{p-1}\nabla u\cdot\nabla w\le
	(3M\xi+\frac{1}{e}M\xi)\intO u^p+M\xi\intO u^pv+2M(p-1)\xi\intO u^{p-1}|\nabla u|
  \eea
  for all $t\in(s_0,T_{\max})$.
\end{lem}
\begin{proof}
  Integration by parts and an application of Lemma \ref{delta w} yield that
  \beaa\nn
	&&(p-1)\xi\intO u^{p-1}\nabla u\cdot\nabla w\\\nn
	&=&-\frac{p-1}{p}\xi\intO u^p\Delta w\\\nn
	&\le& -\frac{p-1}{p}\xi\intO u^p(\Delta w(x,s_0)e^{-\int_{s_0}^tv(x,s)ds}
	-2e^{-\int_{s_0}^t v(x,s)ds}\nabla w(x,s_0)\cdot\int_{s_0}^t\nabla v(x,s)ds\\\nn
	&&~~~~~-\frac{1}{e}w(x,s_0)-w(x,s_0)v(x,t)e^{-\int_{s_0}^tv(x,s)ds})\\\nn
	&\le& (M\xi+\frac{1}{e}M\xi)\intO u^p+M\xi\intO u^pv
	-2\frac{p-1}{p}\xi\intO u^p\nabla w(x,s_0)\cdot\nabla e^{-\int_{s_0}^t
	v(x,s)ds}\\\nn
	&=&(M\xi+\frac{1}{e}M\xi)\intO u^p+M\xi\intO u^pv+2\frac{p-1}{p}\xi\intO u^p\Delta w(x,s_0)e^{-\int_{s_0}^t
	v(x,s)ds}\\\nn
	&&~~~~~~~~~~~~~~+2\frac{p-1}{p}\xi\intO \nabla u^p\cdot \nabla w(x,s_0)
	e^{-\int_{s_0}^t
	v(x,s)ds}\\\nn
	&\le&
	(3M\xi+\frac{1}{e}M\xi)\intO u^p+M\xi\intO u^pv+2M(p-1)\xi\intO u^{p-1}|\nabla u|
  \eea
  for all $t\in(s_0,T_{\max})$, where in accordance with (\ref{s0}),
  $M$ is an upper bound for $\|w(\cdot,s_0)\|_{W^{2,\infty}(\Om)}$.
\end{proof}

\begin{lem}\label{lem2.4}
  Let $\chi>0$, $\xi>0$ and $\mu>0$, and assume \eqref{inidata}.
  Then there exists $C(\mu,|\Om|)>0$ such that
  \begin{align}\label{2.1.0}
	&\intO u(\cdot,t)<C,\;\;\intO v(\cdot,t)<C
  \end{align}
  for all $t\in(0,T_{\max})$.
\end{lem}
\begin{proof}
  The first inequality in (\ref{2.0}) can be proved by simply integrating the first equation in (\ref{2.0}) on $\Om$ and using
  that $(\intO u)^2\le |\Om|(\intO u^2)$ due to the Cauchy-Schwarz inequality.
  The estimate of $\intO v$ can be obtained in a similar way and with the aid of the first inequality.
\end{proof}
As an essential ingredient of the proof of our main result, we will use a Maximal Sobolev regularity property
associated with the second equation in (\ref{2.0}).
The following lemma is not
the original version of a corresponding statement in \cite[Theorem 3.1]{MJ}, but by means of a simple transformation
it is adapted to the current situation by including an exponential weight function.
\begin{lem}\label{maximal}
  Let $r\in (1,\infty)$, $T\in(0,\infty)$. Consider the following evolution equation
  \begin{equation}
	\left\{
	\begin{array}{llc}
	v_t=\Delta v-v+f,&(x,t)\in\Om\times(0,T),\\[6pt]
	\frac{\partial v}{\partial \nu}=0,&(x,t)\in\partial\Om\times(0,T),\\[6pt]
	v(x,0)=v_0(x),&x\in\Om.
	\end{array}
	\right.
  \end{equation}
  For each $v_0\in W^{2,r}(\Omega)$ such that $\frac{\partial v_0}{\partial \nu}=0$ on $\partial \Omega$
  and any $f\in L^r((0,T);L^r(\Omega))$, there exists a unique solution
  $$v\in W^{1,r}((0,T);L^r(\Omega))\cap L^r((0,T);W^{2,r}(\Omega)),$$
  and there exists $C_r>0$ such that
  \bea{est_MJ}
	\int_0^T\intO e^{\frac12rs}|\Delta v(x,s)|^rdxds
	\le C_r\int_0^T\intO e^{\frac12rs}|f(x,s)|^rdxds
	+C_r\|v_0\|_{L^r(\Omega)}^r+C_r\|\Delta v_0\|_{L^r(\Omega)}^r.
  \eea
  Moreover, if for some $s_0\in(0,T)$, $v(\cdot,s_0)$ satisfies $v(\cdot,s_0)\in W^{2,r}(\Omega)$
  with $\frac{\partial v}{\partial \nu}(\cdot,s)=0$ on $\partial\Omega$, then with the same
  constant $C>0$ as above we have
  \bea{est_MJ1}
	\int_{s_0}^T\intO e^{\frac12 rs}|\Delta v(x,s)|^rdxds
	\le C_r\int_{s_0}^T\intO e^{\frac12 rs}|f(x,s)|^rdxds
	+C_r\|v(\cdot,s_0)\|_{L^r(\Omega)}^r+C_r\|\Delta v(\cdot,s_0)\|_{L^r(\Omega)}^r.
  \eea
\end{lem}
\begin{proof}
  Letting $w(x,s)=e^{\frac12 s}v(x,s)-\chi(s)v_0$, where $\chi(s)$ is a cut-off function such that $\chi\in C_0^\infty([0,1))$, and $\chi(0)=1$, we see that $w$ solves
  \begin{equation*}\left\{
	\begin{array}{llc}
	&w_s(x,s)=(\Delta-\frac12) w(x,s)+e^{\frac12s}f(x,s)+\chi(s)\Delta v_0-\frac12\chi(s)v_0-\chi'(s)v_0,&(x,s)\in\Om\times(0,T),\\
	&\frac{\pa w}{\pa\nu}=0,&(x,s)\in\pa\Om\times[0,T),\\
	&w(x,0)=0,&x\in\Om.
	\end{array}
	\right.
  \end{equation*}
Let $A$ be $\Delta-\frac12$ associated with Neumann boundary condition, we know that $A$ is a generator of an analytic semigroup of negative exponential type. An application of standard results on maximal sobolev regularity provides $c_r>0$ such that
  \begin{align}\nn
	&\int_0^T\norm[L^r]{w(\cdot,s)}^r ds
+\int_0^T\norm[L^r]{w_s(\cdot,t)}^r ds
+\int_0^T\norm[L^r(\Om)]{A w(\cdot,s)}^rds\\\nn
&\le c_r\int_0^T\norm[L^r]{e^{\frac12 t}f(\cdot,t)}^r +\norm[L^r]{\chi(s)\Delta v_0-\frac12\chi(s)v_0-\chi'(s)v_0}^r ds\\
&\le c_r\int_0^T\norm[L^r]{e^{\frac12 t}f(\cdot,t)}^r +c_rc_1\norm[W^{2,r}(\Om)]{v_0},\label{msr}
  \end{align}
  where $c_1:=2+\mathop{\sup}\limits_{s\in[0,1]} \chi'(s) $.
  The triangle inequality with \eqref{msr} implies that
\begin{align}\nn
\int_0^T\norm[L^r]{\Delta w(\cdot,s)}^r ds &\le \int_0^T\norm[L^r]{(\Delta-\frac12)w(\cdot,t)}^r ds+\int_0^T \frac12\norm[L^r]{w(\cdot,t)}^r ds\\\nn
&\le c_r \int_0^T\norm[L^r]{e^{\frac12 t}f(\cdot,s)}^r ds+c_1c_r\norm[W^{2,r}]{v_0}^r,
\end{align}
which leads to (\ref{est_MJ}) upon direct computation and letting $C_r:=\max\{c_r,c_1c_r\}$. Thereafter, (\ref{est_MJ1})
  directly follows upon replacing $v(x,t)$ by $v(x,t+s_0)$, where the constant might be different from
  the aforementioned one.
\end{proof}
%
%
%
%
%
%
%
%
%
%
%
%
\section{Boundedness}
In this section, we derive the claimed boundedness result via combining the above result on maximum Sobolev regularity
with a Moser-type iteration.
The former ingredient is first used to estimate $u$ in some appropriate Lebesgue space,
from which a certain suitable estimate of $\nabla v$ will follow. This approach will be carried out twice
to ensure that $\nabla v$ is bounded in $L^\infty(\Om)$. Thereupon we can establish a series
of inequalities based on which a Moser iteration is performed to finally achieve boundedness of $u$ in $L^\infty(\Om)$.
\begin{lem}\label{lem3.0}
  Let $\chi>0$, $\xi>0$, $\mu>0$, $\theta=\frac{\chi}{\mu}$ and let $p_1\in(\frac32,2)$,
   $p_2\in(3,\infty)$. Then
  there exists $\theta_0>0$
  such that whenever $\theta<\theta_0$, for any $(u_0,v_0,w_0)$ fulfilling (\ref{inidata}) and some $s_0\in(0,T_{\max})$,
   we can find $C>0$ such that
  \bea{3.1.0}
	\intO u^{p_2}(\cdot,t)\le C\;\;\text{ for all } t\in(s_0,T_{\max}).
  \eea
\end{lem}
We are going to prove Lemma \ref{lem3.0} by several steps. Let us first provide an important ingredient for the estimate of $\|u(\cdot,t)\|_{L^p(\Om)}$. The next Lemma offers a general iteration step from $v$ to $u$. We will use it to first estimate $\|u(\cdot,t)\|_{L^{p_1}(\Om)}$ with $p_1\in(\frac32,2)$, and then $\|u(\cdot,t)\|_{L^{p_2}(\Om)}$ with $p_2\in(3,\infty)$. These results are under different conditions on $\frac{\chi}{\mu}$. For convenience, we assume that both of the conditions hold in Lemma \ref{lem3.0} such that Lemma \ref{lem3.1} is applicable for both $p_1$ and $p_2$.

\begin{lem}\label{lem3.1}
Let $\chi>0$, $\xi>0$, $\mu>0$, $p>1$. Then there exist constants $\theta_p>0$ and $C>0$ such that if $\theta=\frac{\chi}{\mu}<\theta_p$ and
\beaa\label{vp+1}
\intO v^{p+1}(\cdot,t)\le C\;\;\text{ for all }t\in(s_0,T_{\max}).
\eea
Then we have
\bea{3.3.0}
\intO u^{p}(\cdot,t)\le C\;\;\text{ for all }t\in(s_0,T_{\max}).
\eea
\end{lem}

\begin{proof}
First we see that for any $a,b>0$, Young's inequality provides $k_p>0$ such that
  \begin{align}\label{young}
	ab\le\frac18 a^\frac{p+1}{p}+k_p b^{p+1}.
  \end{align}
  Let $C_{r+1}$ denote the constant from Lemma \ref{maximal} for $r\in(1,\infty)$.
  Now we can find $\theta_p>0$ small enough such that
  \begin{align}\label{thetap}
  C_{p+1}k_{p}\theta^{p+1}\le \frac12 \text{ for all } \theta<\theta_p.
  \end{align}
  Testing the first equation in \eqref{2.0} with $u^{p-1}$ ($p>1$) and integrating by part imply
  \begin{align}\nn
	&~~~~\frac{1}{p}\frac{d}{dt}\intO u^p+(p-1)\intO u^{p-2}|\nabla u|^2\\\nn
	&=(p-1)\chi\intO u^{p-1}\nabla u\cdot\nabla v
	+(p-1)\xi\intO u^{p-1}\nabla u\cdot\nabla w\\\nn
	&~~~~~~~~~~~~~~~~+\mu\intO u^p
	-\mu\intO u^{p+1}-\mu\intO u^pw\\\nn
	&\le \frac{p-1}{p}\chi\intO\nabla u^p\cdot\nabla v
	+(p-1)\xi\intO u^{p-1}\nabla u\cdot\nabla w+\mu\intO u^p-\mu\intO u^{p+1}\\\label{3.1.2}
	&\le-\frac{p-1}{p}\chi\intO u^p\Delta v+(p-1)\xi\intO u^{p-1}\nabla u\cdot\nabla w+\mu\intO u^p-\mu\intO u^{p+1}
  \end{align}
  for all $t\in(s_0,T_{\max})$. We see that \eqref{2.3} and \eqref{3.1.1} entail the existence of
  $c_3>0$
  satisfying
  \begin{align}\nn
	&~~~~(p-1)\xi\intO u^{p-1}\nabla u\cdot\nabla w\\\nn
	&\le c_3\xi\intO u^p+c_3\xi\intO u^pv+c_3p\xi\intO u^{p-1}|\nabla u|\\\nn
	&\le c_3\xi\intO u^p+\frac{\mu}{8}\intO u^{p+1}+k_pc_3^{p+1}\mu^{-p}\xi^{p+1}\intO v^{p+1}
	+\frac{p-1}{2}\intO u^{p-2}|\nabla u|^2+\frac{c_3^2\xi^2p^2}{2(p-1)} \intO u^p\\\label{3.1.3}
	&\le (c_3\xi+\frac{c_3^2\xi^2p^2}{2(p-1)})\intO u^p+\frac{\mu}{8}\intO u^{p+1}+\frac{p-1}{2}\intO u^{p-2}|\nabla u|^2+k_p\mu^{-p}\xi^{p+1}\|v(\cdot,t)\|_{L^{p+1}(\Om)}^{p+1}
  \end{align}
  for all $t\in(s_0,T_{\max})$.
  From \eqref{young}, we estimate for all $t\in(s_0,T_{\max})$,
  \begin{align}\label{3.1.4}
	-\frac{p-1}{p}\chi\intO u^p \Delta v \le \chi\intO u^p|\Delta v|\le \frac{\mu}{8}\intO u^{p+1}
	+k_p\chi^{p+1}\mu^{-p}\intO|\Delta v|^{p+1}.
  \end{align}
Inserting (\ref{3.1.3}) and (\ref{3.1.4}) into \eqref{3.1.2} and some rearrangement yield
  \begin{align}\nn
	&~~~~\frac{1}{p}\frac{d}{dt}\intO u^p+\frac{p-1}{2}\intO u^{p-2}|\nabla u|^2\\\nn
	&\le -\frac{3}{4}\mu\intO u^{p+1}+\left(c_3\xi+\frac{c_3^2\xi^2p^2}{2(p-1)}+\mu\right)\intO u^p
	+k_p\mu^{-p}\chi^{p+1}\intO |\Delta v|^{p+1}
	+ k_p\mu^{-p}\xi^{p+1}\|v(\cdot,t)\|_{L^{p+1}(\Om)}^{p+1}\\\nn
	&= -\frac{p+1}{2p}\intO u^p+\left(\frac{p+1}{2p}+c_3\xi+\frac{c_3^2\xi^2p^2}{2(p-1)}+\mu\right)\intO u^p-\frac{3}{4}\mu\intO u^{p+1}
	\\\label{3.1.6}
    &~~~~~~~~~~~~~~~~~~~~~~~~~~~+k_p\chi^{p+1}\mu^{-p}\intO |\Delta v|^{p+1}+k_p\mu^{-p}\xi^{p+1}\|v(\cdot,t)\|_{L^{p+1}(\Om)}^{p+1}
  \end{align}
  for all $t\in(s_0,T_{\max})$. We again apply Young's inequality to obtain that
  \beaa\label{3.1.7}
	\left(\frac{p+1}{2p}+c_3\xi+\frac{c_3^2\xi^2p^2}{2(p-1)}+\mu\right)\intO u^p\le\frac{\mu}{4}\intO u^{p+1}+c_4(\mu,\xi,p),
  \eea
  where $c_5>0$
  depends on $\mu,\xi,p$. Upon \eqref{3.1.6} and \eqref{3.1.7}, we infer that
  \begin{align}\nn
	\frac d{dt}\intO u^p&\le -\left(\frac{p+1}2\right)\intO u^p-\frac{\mu}{2}p\intO u^{p+1}
	+p k_p\chi^{p+1}\mu^{-p}\intO |\Delta v|^{p+1}\\\nn
    &~~~~~~~~~~~~~~~~~~~~~+c_4(\mu,\xi,p)p+p k_p\mu^{-p}\xi^{p+1}\|v(\cdot,t)\|_{L^{p+1}(\Om)}^{p+1}
  \end{align}
  for all $t\in(s_0,T_{\max})$. Applying Gronwall inequality in different form to the above inequality shows that
  \begin{align}\nn
	\intO u^p(\cdot,t)&\le e^{-\left(\frac{p+1}2\right)(t-s_0)}\intO u^p(\cdot,s_0)-\frac{\mu}{2}p\int_{s_0}^te^{-\left(\frac{p+1}2\right)(t-s)}\intO u^{p+1}(\cdot,s)ds\\\nn
	&~~~~+kp\chi^{p+1}\mu^{-p}\int_{s_0}^t e^{-\left(\frac{p+1}2\right)(t-s)}\intO |\Delta v(\cdot,s)|^{p+1}ds+\int_{s_0}^t e^{-\left(\frac{p+1}2\right)(t-s)}(c_4(\mu,\xi,p)\\\nn
&~~~~~~~~~~~~+k_p\mu^{-p}\xi^{p+1}\|v(\cdot,t)\|_{L^{p+1}(\Om)}^{p+1})ds\\\nn
	&= e^{-\left(\frac{p+1}2\right)(t-s_0)}\intO u^p(\cdot,s_0)-\frac{\mu}{2}pe^{-\left(\frac{p+1}2\right)t}\int_{s_0}^t\intO e^{\left(\frac{p+1}2\right)s}u^{p+1}(\cdot,s)\\\label{3.1.8}
	&~~~~+pk_p\chi^{p+1}\mu^{-p}e^{-\left(\frac{p+1}2\right)t}\int_{s_0}^t\intO e^{\left(\frac{p+1}2\right)s}|\Delta v(\cdot,s)|^{p+1}\\\nn
&~~~~~~~~~~~~+\frac{p}{p+1}\left(c_4(\mu,\xi,p)+k_p\mu^{-p}\xi^{p+1}\|v(\cdot,t)\|_{L^{p+1}(\Om)}^{p+1}\right)
  \end{align}
  for all $t\in({s_0},T_{\max})$. In order to estimate the third term therein, let us note that
   an application of Lemma \ref{maximal} results in
  \begin{align}\nn
	&~~~~p k_p \chi^{p+1}\mu^{-p}e^{-\left(\frac{p+1}2\right)t}\int_{s_0}^t \intO e^{\left(\frac{p+1}2\right)s} |\Delta v|^{p+1}\\\nn
	&\le pk_pC_{p+1}\chi^{p+1}\mu^{-p}e^{-\left(\frac{p+1}2\right)t}\int_{s_0}^t\intO e^{\left(\frac{p+1}2\right)s} u^{p+1} +C_{p+1}pk_p\chi^{p+1}\mu^{-p}e^{-\left(\frac{p+1}2\right)t}
\|v(\cdot,s_0)\|_{W^{2,p+1}(\Omega)}^{p+1}\\\label{3.1.9}	&=pk_pC_{p+1}\chi^{p+1}\mu^{-p}e^{-\left(\frac{p+1}2\right)t}\int_{s_0}^t\intO e^{\left(\frac{p+1}2\right)s} u^{p+1} +C_{p+1}pk_p\chi^{p+1}\mu^{-p}e^{-\left(\frac{p+1}2\right)t}M^{p+1}
  \end{align}
  for all $t\in(s_0,T_{\max})$  and $M$ as in (\ref{s0}).
  Since $\theta=\frac{\chi}{\mu}$, combining (\ref{3.1.8}-\ref{3.1.9}), we finally arrive at
  \beaa\nn
	\intO u^p(\cdot,t)&\le& e^{-\left(\frac{p+1}2\right)(t-s_0)}\intO u^p(\cdot,s_0)-\mu p\left(\frac{1}{2}
	-k_pC_{p+1}\theta^{p+1}\right)e^{-\left(\frac{p+1}2\right)t}\int_{s_0}^t\intO e^{\left(\frac{p+1}2\right)s} u^{p+1}\\\nn &&~~~~+C_{p+1}pk_p\chi^{p+1}\mu^{-p}e^{-\left(\frac{p+1}2\right)t}M^{p+1}\\\label{3.2.0}
&&~~~~~~~~~~+\frac{p}{p+1}\left(c_4(\mu,\xi,p)
+k_p\mu^{-p}\xi^{p+1}\|v(\cdot,t)\|_{L^{p+1}(\Om)}^{p+1}\right)
  \eea
  for all $t\in(s_0,T_{\max})$. We see from (\ref{vp+1}) and the condition on $\theta_p$ in \eqref{thetap}
  that
  \beaa
	\intO u^{p}(\cdot,t) \le C(\mu,\chi,\xi,p,\|v(\cdot,s_0)\|_{W^{2,p+1}(\Omega)})
  \eea
  for all $t\in(s_0,T_{\max})$ upon obvious choice of $C>0$. Thus the assertion is derived.
\end{proof}
Now we can apply Lemma \ref{lem3.1} to improve the regularity of $u$ from $L^1(\Om)$ to $L^{p_1}(\Om)$.
\begin{lem}\label{lem3.2}
Let $p_1\in(\frac 32,2)$ and $\theta<\theta_{p_1}$, there is a constant $C>0$ such that
\begin{align}\label{up1}
\|u(\cdot,t)\|_{L^{p_1}(\Om)}\le C\;\;\text{ for all } t\in(s_0,T_{\max}).
\end{align}
\end{lem}
\begin{proof}
First we use the representation formula for $v$,
  \begin{align}
  v(\cdot,t)=e^{t\Delta}v_0+\int_{0}^t e^{(t-s)(\Delta-1)}u(\cdot,s)ds
  \end{align}
  for all $t\in(0,T_{\max})$. Let $p_1\in(\frac32,2)$, the $L^p-L^q$ estimates for the Neumann heat
  semigroup  \cite[Lemma1.3]{W4} and (\ref{2.1.0}) allow us to pick constants $c_1>0$ and $c_2>0$ such that
  \begin{align}\nn
	\|v(\cdot,t)\|_{L^{p_1+1}(\Om)}&\le \|e^{t(\Delta-1)}v_0\|_{L^{p_1+1}(\Om)}+\int_0^t\|e^{(t-s)(\Delta-1)}u
	(\cdot,s)\|_{L^{p_1+1}(\Om)}ds\\\nn
	&\le e^{-t}\|v_0\|_{L^{p_1+1}(\Om)}+c_1\int_0^t(1+(t-s)^{-\frac{3}{2}(1-\frac{1}{p_1+1})})e^{-(t-s)}
	\|u(\cdot,s)\|_{L^1(\Om)}ds\\\label{3.1.1}
	&\le c_2
  \end{align}
  for all $t\in(0,T_{\max})$. An application of Lemma \ref{lem3.1} and (\ref{3.1.1}) lead to (\ref{up1}).
\end{proof}
With the higher regularity for $v$ obtained in the above lemma, a similar reasoning for Lemma \ref{lem3.1}
will provide us higher regularity for $u$.
Now we are ready to prove Lemma \ref{lem3.0}.

\begin{proof}[Proof of Lemma \ref{lem3.0}]
Let $p_1$ be choosen as in Lemma \ref{lem3.1} and $p_2\in(3,\infty)$, define $\theta_0=\min\{\theta_{p_1},\theta_{p_2}\}$.
According to Lemma \ref{lem3.1} and standard estimates for Neumann heat semigroup generated by $\Delta$,
we see that
\begin{align}\nn
\|v(\cdot,t)\|_{L^\infty(\Om)}&\le \|e^{t(\Delta-1)}v(\cdot,s_0)\|_{L^\infty(\Om)}
+\int_0^t\|e^{(t-s)(\Delta-1)}u(\cdot,s)\|_{L^\infty(\Om)}ds\\\label{v}
&\le e^{-t}\|v(\cdot,s_0)\|_{L^\infty(\Om)}+c_1\int_{s_0}^t(1+(t-s)^{-\frac 3{2p_1}})
e^{-(t-s)}\|u(\cdot,s)\|_{L^{p_1}(\Om)}ds
\end{align}
for all $t\in(s_0,T_{\max})$ with some $c_1>0$. Lemma \ref{lem3.2} guarantees the boundedness of the above estimate. Also H\"older's inequality yields that
\bea{vinfty}
\|v(\cdot,t)\|_{L^{p_2+1}(\Om)}<c_1, \;\;\;\; t\in(s_0,T_{\max}).
\eea
Due to the definition of $\theta_0$ and (\ref{vinfty}), now we can apply Lemma \ref{lem3.1} to find the existence of a constant
$C(\mu,\chi,\xi,p_2,\|v(\cdot,s_0)\|_{W^{2,p_2+1}(\Omega)})>0$ such that
\beaa\nn
\intO u^{p_2}(\cdot,t)\le C(\mu,\chi,\xi,p_2,\|v(\cdot,s_0)\|_{W^{2,p_2+1}(\Omega)})
\eea
for all $t\in({s_0},T_{\max})$. Thus \eqref{3.3.0}
 is obtained.
\end{proof}
An immediate consequence of Lemma \ref{lem3.0} is that $\nabla v$ is bounded with respect to the norm in $L^\infty(\Om)$.

\begin{lem}\label{lem3.4}
Under the assumptions of Lemma \ref{lem3.1}, for some $s_0\in(0,T_{\max})$, there is $C>0$ such that
\begin{align}\label{nablav}
\|\nabla v(\cdot,t)\|_{L^\infty(\Om)}<C\;\;\text{ for all } t\in(s_0,T_{\max}).
\end{align}
\end{lem}
\begin{proof}
We note that the standard estimate for Neumann semigroup provides $c_1>0$ such that
\begin{align}\nn
\|\nabla v(\cdot,t)\|_{L^\infty(\Om)}&\le\|\nabla e^{t(\Delta-1)}v(\cdot,s_0)\|_{L^\infty(\Om)}
+\int_{s_0}^t\|\nabla e^{(t-s)(\Delta-1)}u(\cdot,s)\|_{L^\infty(\Om)}\\\nn
&\le e^{-t}\|\nabla v(\cdot,s_0)\|_{L^\infty(\Om)}+\int_{s_0}^t c_1(1+(t-s)^{-\frac{1}{2}-\frac{3}{2p_2}})e^{-(t-s)}\|u(\cdot,s)\|_{L^{p_2}(\Om)}ds
\end{align}
for all $t\in(0,T_{\max})$. In light of Lemma \ref{lem3.0}, we see the conclusion is established.
\end{proof}

Now we are in a position to prove boundednes of $u$. Since the estimate of $\nabla w$ is nonlocal in time, the Moser iteration procedure is still necessary. Upon which,
we can finally show that $u$ is bounded.
\begin{lem}\label{lem3.5}
Under the assumptions of Lemma \ref{lem3.1}, for some $s_0\in(0,T_{\max})$, there is $C>0$ such that
\begin{align}
\|u(\cdot,t)\|_{L^\infty(\Om)}\le C\;\;\text{ for all } t\in(s_0,T_{\max}).
\end{align}
\end{lem}
\begin{proof}
We first see Lemma \ref{lem3.2} and Lemma \ref{lem3.4} imply the existence of $c_1>0$ such that
\begin{align}
\|v(\cdot,t)\|_{L^\infty(\Om)}+\|\nabla v(\cdot,t)\|_{L^\infty(\Om)}<c_1
\end{align}
for all $t\in(s_0,T_{\max})$. Testing the first equation in \eqref{2.0} with $u^{p-1}$ ($p>1$), using (\ref{2.3}), (\ref{nablav}), (\ref{v})
 and Young's inequality, we can find constants $c_2,c_3>0$ such that
\begin{align}\nn
&~~~~\frac1p\frac{d}{dt}\intO u^p+(p-1)\intO u^{p-2}|\nabla u|^2\\\nn
&=(p-1)\chi\intO u^{p-1}\nabla u\cdot\nabla v
+(p-1)\xi\intO u^{p-1}\nabla u\cdot\nabla w+\mu\intO u^p-\mu\intO u^{p+1}-\mu\intO u^pw\\\nn
&\le (p-1)\chi\intO u^{p-1}\nabla u\cdot\nabla v
+c_2\xi\intO u^p+c_2\xi\intO u^pv+c_2\xi\intO u^{p-1}|\nabla u|+\mu\intO u^p\\\nn
&\le \frac{p-1}{4}\intO u^{p-2}|\nabla u|^2+(p-1)\chi^2\intO u^p|\nabla v|^2+c_2\xi\intO u^p
+c_1c_2\xi\intO u^p\\\nn
&~~~~~~~~~~~~+\frac{p-1}{4}\intO u^{p-2}|\nabla u|^2+c_2^2\xi^2p\intO u^p+\mu\intO u^p\\\nn
&\le \frac{p-1}{2}\intO u^{p-2}|\nabla u|^2+c_3p\intO u^p
\end{align}
for all $t\in(s_0,T_{\max})$. An obvious rearrangement implies
\begin{align}\label{3.5.1}
\frac{d}{dt}\intO u^p+c_4\intO |\nabla u^{\frac{p}{2}}|^2\le c_5 p^2\intO u^p
\end{align}
for all $t\in(s_0,T_{\max})$ where $c_4,c_5>0$ are independent of $p$.
Next, we use (\ref{3.5.1}) to perform the classical
 Moser iteration procedure \cite{alikakos} to
obtain the boundedness of $u$.

Let $p_k=2^k$, $k\in \mathbb{N}$ and $M_k:=\mathop{\sup}\limits_{t\in(s_0,T_{\max})}\intO u^{p_k}(\cdot,t)$.  Since $p_k\ge 1$,
 it is easy to find $c_6>0$ such that
\begin{align}\label{th1.1}
\frac{d}{dt}\intO u^{p_k}+\intO u^{p_k}+c_4\intO|\nabla u^{\frac{p_k}{2}}|^2\le c_5{p_k}^2\intO u^{p_k}
+\intO u^{p_k}
\le c_6{p_k}^2\intO u^{p_k}
\end{align}
for all $t\in(s_0,T_{\max})$. By means of the Gargliardo-Nirenberg inequality, we get that
\begin{align}\nn
\intO u^{p_k}=\|u^{\frac{p_k}{2}}\|_{L^2(\Om)}^2\le c_7\|\nabla u^{\frac{p_k}{2}}\|_{L^2(\Om)}^{2a}
\|u^{\frac{p_k}{2}}\|_{L^1(\Om)}^{2(1-a)}+c_7\|u^{\frac{p_k}{2}}\|_{L^1(\Om)}^2
\end{align}
with $a=\frac{\frac{N}{2}}{1+\frac{N}{2}}\in (0,1)$ and some $c_7>0$ independent of $k$.
Young's inequality and the definition of $p_k$ ensure that there are $c_8>0$ and $b>0$
satisfying
\begin{align}\nn
c_6p_k^2\intO u^{p_k}&\le c_4\intO|\nabla u^{\frac{p_k}{2}}|^2+c_8\left(p_k^2\right)^{\frac{1}{1-a}}
\left(\intO u^{p_{k-1}}\right)^2
+c_6c_7p_k^2\left(\intO u^{p_{k-1}}\right)^2\\\label{th1.2}
&\le c_4\intO|\nabla u^{\frac{p_k}{2}}|^2+b^k M_{k-1}^2.
\end{align}
Combining (\ref{th1.1}-\ref{th1.2}) we find that
\begin{align}\nn
\frac{d}{dt}\intO u^{p_k}+\intO u^{p_k}\le b^k M_{k-1}^2.
\end{align}
for all $t\in(s_0,T_{\max})$. The comparison theorem for the above ODE yields
\begin{align}\nn
M_k\le \max\{b^k M_{k-1}^2,\intO u^{p_k}(\cdot,s_0)\}.
\end{align}
If $b^k M_{k-1}^2<\intO u^{p_k}(\cdot,s_0)$ is valid for infinite $k$, (\ref{lem3.5}) is already derived.
 Otherwise, a direct induction entails
 \begin{align}\nn
 M_k\le b^k M_{k-1}^2\le b^{{\sum_{j=1}^{j=k-1}}2^j(k-j)}M_0^{2k}.
 \end{align}
Taking $2^k$-th root on both sides leads to the assertion.
\end{proof}
Now we are ready to prove Theorem \ref{th1}.
\begin{proof}[Proof of Theorem \ref{th1}]
First we see that the boundedness of $u$ and $v$ follow from Lemma \ref{lem3.5},
Lemma \ref{lem3.2} and (\ref{s0}). Thereupon the assertion of Theorem \ref{th1}
is immediately
obtained from Lemma \ref{loc}.
\end{proof}

\section{Acknowledgement}
I would like to thank two anonymous referees for carefully reading and giving many suggestions to improve the manuscript.

\end{document}